\documentclass[reqno]{amsart}
\usepackage{amssymb}
\usepackage{amscd}
\usepackage[dvips]{graphics}
\usepackage{mathrsfs}
\usepackage{bm}
\usepackage{cite}

\def\beq{\begin{equation}}
\def\eeq{\end{equation}}
\def\ba{\begin{array}}
\def\ea{\end{array}}

\numberwithin{equation}{section}
\newtheorem{theorem}{Theorem}[section]

\newtheorem{proposition}[theorem]{\textbf{Proposition}}
\newtheorem{lemma}[theorem]{Lemma}
\renewenvironment{proof}{\noindent{\textbf{Proof.}}}{\hfill$\Box$}
\theoremstyle{remark}
\newtheorem{remark}[theorem]{\textbf{Remark}}
\theoremstyle{plain}

\allowdisplaybreaks[4]

\begin{document}
\title[HLS inequalities on compact CR manifold]{\textbf{Sharp Hardy-Littlewood-Sobolev inequlities on compact CR manifold}}
\author  {Yazhou Han}
\address{Yazhou Han, Department of Mathematics, College of Science, China Jiliang University, Hangzhou, 310018, China}
\email{yazhou.han@gmail.com}

\today

\date{}


\begin{abstract}
  Assume that $M$ is a CR compact manifold without boundary and CR Yamabe invariant $\mathcal{Y}(M)$ is positive. Here, we devote to study a class of sharp Hardy-Littlewood-Sobolev inequality as follows
  \begin{equation*}
    \Bigl| \int_M\int_M [G_\xi^\theta(\eta)]^{\frac{Q-\alpha}{Q-2}} f(\xi) g(\eta) dV_\theta(\xi) dV_\theta(\eta) \Bigr| \leq \mathcal{Y}_\alpha(M) \|f\|_{L^{\frac{2Q}{Q+\alpha}}(M)} \|g\|_{L^{\frac{2Q}{Q+\alpha}}(M)},
  \end{equation*}
  where $G_\xi^\theta(\eta)$ is the Green function of CR conformal Laplacian $\mathcal{L_\theta}=b_n\Delta_b+R$, $\mathcal{Y}_\alpha(M)$ is sharp constant, $\Delta_b$ is Sublaplacian and $R$ is Tanaka-Webster scalar curvature. For the diagonal case $f=g$, we prove that $\mathcal{Y}_\alpha(M)\geq \mathcal{Y}_\alpha(\mathbb{S}^{2n+1})$ (the unit complex sphere of $\mathbb{C}^{n+1}$) and $\mathcal{Y}_\alpha(M)$ can be attained if $\mathcal{Y}_\alpha(M)> \mathcal{Y}_\alpha(\mathbb{S}^{2n+1})$. Particular, if $\alpha=2$, the previous extremal problem is closely related to the CR Yamabe problem. Hence, we can study the CR Yamabe problem by integral equations.
\end{abstract}
\maketitle

\section{Introduction}

CR geometry, the abstract models of real hypersurfaces in complex manifolds, has attracted much attention in the past decades. Noticing that there is a far-reaching analogy between conformal and CR geometry, such as Model space, scalar curvature, Sublaplacian and Yamabe equation etc., many interesting and profound results on CR geometry were obtained, see \cite{CMY2017, Dragomir-Tomassini2006, Folland1973, Folland1975, Folland-Stein1974, Frank-Lieb2012, JL1983, JL1987, JL1988, JL1989, BFM2013, Cohn-Lu2004, Han-arxiv, Hormander1967, Lee1986, Lee1988, Li-Son-Wang2015, Li-Wang2015, Wang, Wang2012, Gamara2001, Gamara-Yacoub2001} 
and the references therein. Inspired by the idea of \cite{Zhu2016, Han-Zhu2016, Dou-Zhu, Han-arxiv}, we want to study the curvature problem of CR geometry from the point of integral curvature equation. Following, involved notations can be found in the Section \ref{Sec Preliminary}.


\medskip

Let $(M,J,\theta)$ be a compact psedudohermitian manifold without boundary. Under the transformation $\tilde{\theta}=\phi^{\frac{4}{Q-2}}\theta$ with $\phi\in C^\infty(M)$ and $\phi>0$, Tanaka-Webster scalar curvatures $R$ and $\tilde{R}$, corresponding to $\theta$ and $\tilde{\theta}$ respectively, satisfy
\begin{equation}\label{CR Yamabe equation}
    b_n\Delta_b \phi+R\phi=\tilde{R}\phi^{\frac{Q+2}{Q-2}},
\end{equation}
where $\mathcal{L}_\theta= b_n\Delta_b+R$ is the CR conformal Laplacian related to $\theta$. 
For given constant curvature $\tilde{R}$, the existence of \eqref{CR Yamabe equation} is known as CR Yamabe problem, which was introduced by Jerison and Lee in \cite{JL1987}. There, they studied the CR Yamabe invariant
    $$\mathcal{Y}(M)=\inf\{A_\theta(\phi): B_\theta(\phi)=1\}$$
with
    $$A_\theta(\phi)=\int_M (b_n |d\phi|_\theta^2 +R \phi^2)\ dV_\theta, \quad B_\theta(\phi)= \int_M |\phi|^p\ dV_\theta, \quad dV_\theta=\theta \wedge d\theta^n,$$
and proved that $\mathcal{Y}(M) \leq \mathcal{Y}(\mathbb{S}^{2n+1})$ and the infimum can be attained if $\mathcal{Y}(M)<\mathcal{Y}(\mathbb{S}^{2n+1})$. It is well known that the case of $\mathcal{Y}(M)\leq 0$ is easy to deal. While for the positive case, it is complicated.

If $\mathcal{Y}(M)>0$, then the first eigenvalue $\lambda_1(\mathcal{L}_\theta)>0$ and then $\mathcal{L}_\theta$ is invertible. Furthermore, for any $\xi\in N$, there exists a Green function $G_\xi^\theta(\eta)$ of $\mathcal{L}_\theta$ such that the solution of \eqref{CR Yamabe equation} satisfies
\begin{equation}\label{CR Yamabe integral equ}
    \phi(\xi)=\int_M G_\xi^\theta(\eta) \tilde{R}(\eta) \phi(\eta)^{\frac{Q+2}{Q-2}}\ dV_\theta.
\end{equation}
So, scalar curvature $\tilde{R}$ can also be defined implicitly by the integral equation \eqref{CR Yamabe integral equ}.

Noting that $G_\xi^{\tilde{\theta}}(\eta) =\phi^{-1}(\xi)\phi^{-1}(\eta)G_\xi^\theta(\eta)$, we easily get
    $$\int_M G_\xi^{\tilde{\theta}}(\eta) \tilde{R}(\eta) u(\eta)\ dV_{\tilde{\theta}} =\phi^{-1} \int_M G_\xi^\theta(\eta) \tilde{R}(\eta) \phi^{\frac{Q+2}{Q-2}}(\eta) u(\eta)\ dV_\theta.$$
Moreover, we can study a class of more general CR conformal integral equation as in \cite{Zhu2016}. Specifically, define the operator
\begin{equation}\label{conformal operator}
    I_{M,\theta,\alpha}(u)=\int_M [G_\xi^\theta(\eta)]^{\frac{Q-\alpha}{Q-2}} u(\eta) \ dV_\theta \quad \text{with}\quad \alpha\neq Q,
\end{equation}
and we can prove that, under the transformation $\tilde{\theta}=\phi^{\frac 4{Q-2}}\theta$,
\begin{equation}\label{conformal operator property}
    I_{M,\tilde{\theta},\alpha}(u) =\phi^{-\frac{Q-\alpha}{Q-2}} I_{M,\theta,\alpha}(\phi^{\frac{Q+\alpha}{Q-2}}u).
\end{equation}
If take $u\equiv C$ and let $\varphi(\xi)=\phi^{\frac{Q+\alpha}{Q-2}}(\xi)$, then we have the following generalizing curvature equation, up to a constant multiplier,
\begin{equation}\label{curvature alpha 2}
    \varphi(\xi)^{\frac{Q+\alpha}{Q-\alpha}}=\int_M [G_\xi^\theta(\eta)]^{\frac{Q-\alpha}{Q-2}} \varphi(\eta)\ dV_\theta,\quad \alpha\neq Q.
\end{equation}
As pointed by Zhu in \cite{Zhu2016}, on $\mathbb{S}^n$ integral curvature equations are equivalent to the classical curvature quation if $\alpha$ is strictly less than dimension; while for the case $\alpha$ strictly greater than dimension, they are not equivalent and integral curvature equation has some advantages. So, it is interesting and valuable to study the integral curvature equation \eqref{curvature alpha 2}.

Moreover, if $G_\xi^\theta(\eta)=G_\eta^\theta(\xi)$, we note that \eqref{curvature alpha 2} is the Euler-Lagrange equation of the extremal problem
\begin{equation}\label{extremal problem}
    \mathcal{Y}_\alpha(M) = \sup_{f\in L^{(2Q)/(Q+\alpha)}(M)\backslash\{0\}} \frac{\Bigl| \int_M\int_M [G_\xi^\theta(\eta)]^{\frac{Q-\alpha}{Q-2}} f(\xi) f(\eta) dV_\theta(\xi) dV_\theta(\eta) \Bigr|}{\|f\|_{L^{(2Q)/(Q+\alpha)}(M)}^2},
\end{equation}
which is closely related to a class of Hardy-Littlewood-Sobolev inequalities with kernel $[G_\xi^\theta(\eta)]^{\frac{Q-\alpha}{Q-2}}$. Namely, for any $f,g\in L^{{2Q}/{Q+\alpha}}(M)$ with $0<\alpha<Q$, there exists some positive constant $C(\alpha, M)$ such that
\begin{equation}\label{roughly HLS diagonal}
    \Bigl| \int_M\int_M [G_\xi^\theta(\eta)]^{\frac{Q-\alpha}{Q-2}} f(\xi) g(\eta) dV_\theta(\xi) dV_\theta(\eta) \Bigr| \leq C(\alpha,M) \|f\|_{L^{\frac{2Q}{Q+\alpha}}(M)} \|g\|_{L^{\frac{2Q}{Q+\alpha}}(M)}.
\end{equation}
In fact, by the parametrix method, we know by \cite{Folland-Stein1974} that $(G^\theta_\xi(\eta))^{\frac{Q-\alpha}{Q-2}}\sim \rho(\xi,\eta)^{\alpha-Q}$ as $\rho(\xi,\eta)\rightarrow 0$. So, \eqref{roughly HLS diagonal} holds by a similar argument with Theorem 15.11 of \cite{Folland-Stein1974}.

In this paper, we will mainly devoted to study the extremal problem \eqref{extremal problem} by Hardy-Littlewood-Sobolev inequalities and will prove the following results.

\begin{theorem}\label{thm main}
Firstly, $\mathcal{Y}_\alpha(M) \geq \mathcal{Y}_\alpha(\mathbb{S}^{2n+1})$. Moreover, if the strict inequality holds, then $\mathcal{Y}_\alpha(M)$ is attained.
\end{theorem}

Because of the hypoellipticity of operator $\mathcal{L}$ (in fact $\mathcal{L}$ satisfies the H\"{o}rmander condition \cite{Hormander1967}), we know that the Green function $G_\xi^\theta(\eta)$ is $C^\infty$ if $\xi\neq\eta$. Moreover, using CR normal coordinates at $\xi$ and the classical method of parametrix, we can construct the Green function as (without loss of generality, we take the coefficient of singular part as one)
    $$G_\xi(\eta)=\rho^{-2n}+w(\xi,\eta),$$
where $w$ is the regular part. Particular, if $M$ is locally CR conformal flat, then $w$ satisfies $\Delta_b w=0$ in some neighbourhood of $\xi$. Therefore, $w$ is $C^\infty$ in this neighbourhood because of the hypoellipticity of $\Delta_b$. If $n=1$, Cheng, Malchiodi and Yang \cite{CMY2017} proved that $w\in C^{1,\gamma}(M)$ for any $\gamma\in (0,1)$. {\bf In the sequel, we always assume that $w(\xi,\eta)\in C^1(M\times M)$.} Then, we can rewrite $G_\xi(\eta)$ as
\begin{equation}\label{Green}
    G_\xi(\eta)=\rho^{-2n}+A(\xi)+O(\rho)\quad\text{with}\quad A(\xi)=w(\xi,\xi).
\end{equation}
On the other hand, on a locally CR conformal flat manifold, we note that $\rho(\xi,\eta)=\rho(\eta,\xi)$ holds on some neighbourhood of diagonal of $M\times M$. So, we give a sufficient condition for the strict inequality of Theorem \ref{thm main}.
\begin{proposition}\label{pro A>0}
Assume that $M$ is a locally CR conformal flat manifold with $\mathcal{Y}(M)>0$. If there exists some point $\xi_0\in M$ such that $A(\xi_0) =w(\xi_0,\xi_0)>0$, then $\mathcal{Y}_\alpha(M)>\mathcal{Y}_\alpha(\mathbb{S}^{2n+1})$.
\end{proposition}

The paper is organized as follows. In Section \ref{Sec Preliminary}, we introduce the definition of CR manifold, some notations and some known results. Section \ref{Sec estimate} is mainly devoted to the first part of Theorem \ref{thm main}, namely, the estimation $\mathcal{Y}_\alpha(M) \geq \mathcal{Y}_\alpha(\mathbb{S}^{2n+1})$. For the discussion of the second part of Theorem \ref{thm main}, we will adopt the blowup analysis. So, we will study the subcritical case of Hardy-Littlewood-Sobolev inequality in the Section \ref{Sec sub}. Then, in Section \ref{Sec existence of criteria}, we complete the proof of Theorem \ref{thm main} and discuss the condition of strict inequality, namely Proposition \ref{pro A>0}. For completeness, we study the CR comformality of operator \eqref{conformal operator} in the Appendix \ref{Sec CR conformality}.

\section{Preliminary}\label{Sec Preliminary}

In this section, we will introduce some notations and some known facts. The details can be seen in \cite{Folland-Stein1974, Dragomir-Tomassini2006, Frank-Lieb2012, JL1983, JL1987, JL1988, JL1989, Lee1986, Lee1988, BFM2013, Li-Son-Wang2015, Li-Wang2015, Wang, Wang2012} and references therein.

\subsection{CR manifold and CR Yamabe problem}
A {\it CR manifold} is a real oriented $C^\infty$ manifold $M$ of dimension $2n+1,\ n=1,2,\cdots$, together with a subbundle $T_{1,0}$ of the complex tangent bundle $\mathbb{C}TM$ satisfying:\par
(a) $\dim_\mathbb{C} T_{1,0}=n,$\par
(b) $T_{1,0}\cap T_{0,1}=\{0\}$ with $T_{0,1}=\overline{T}_{1,0}$,\par
(c) $T_{1,0}$ satisfies the formal Frobenius condition $[T_{1,0},T_{1,0}]\subset T_{1,0}$.\par
\noindent Denote by $Q=2n+2$ the homogeneous dimension.

An \textit{almost CR structure} on $M$ is a pair $(H(M),J)$, where $H(M)=\rm{Re}(T_{1,0}+T_{0,1})$ is a subbundle of rank $2n$ of $T(M)$ and $J:\ H(M)\rightarrow H(M)$, given by $J(V+\overline{V})=\sqrt{-1}(V-\overline{V})$ for $V\in T_{1,0}$,  is an almost complex structre on $H(M)$.

Since $M$ is orientable, then $H(M)$ is oriented by its complex structure. So, there always exists a global nonvanishing $1$-form $\theta$ which annihilates exactly $H(M)$ and for which there exists a natural volume form $dV_\theta=\theta\wedge d\theta^n$. Any such $\theta$ is called a \textit{pseudo-hermitian structure} $M$.

Associated with each $\theta$, {\it Levi form $L_\theta$} is defined on $H(M)$ as
    $$L_\theta(V,W)=\langle d\theta,V\wedge JW\rangle =d\theta(V,JW),\quad V,W\in G.$$
By complex linearity, we can extend $L_\theta$ to $CH(M)$ and induce a hermitian form on $T_{1,0}$ as
    $$L_\theta(V,\overline{W}) =-\sqrt{-1}\langle d\theta,V\wedge\overline{W}\rangle =-\sqrt{-1} d\theta(V,\overline{W}) , \quad V,W\in T_{1,0}.$$
It is easy to see that Levi form is CR invariant. Namely, if $\theta$ is replaced by $\tilde\theta=f\theta$, $L_\theta$ changes conformally by $L_{\tilde\theta}=fL_\theta$. We say $M$ is {\it nondegenerate} if the Levi form is nondegenerate at every point, and say $M$ is {\it strictly pseudoconvex} if the form is positive definite everywhere. In this paper, we always assume that $M$ is strictly pseudoconvex.

Based on the Levi form $L_\theta$, we can take a local unitary frame $\{T_\alpha:\ \alpha=1,\cdots,n\}$ for $T_{1,0}(M)$. Then, there is a natural seconde order differential operator, namely the \textit{Sublaplacian} $\Delta_b$, which is defined on the function $u$ as
\begin{equation}\label{sub-Laplacian}
    \Delta_b u=-(u_{\bar{\alpha},\alpha}+u_{\alpha,\bar{\alpha}}).
\end{equation}
Under the transformation $\tilde\theta=\phi^{\frac 4{Q-2}}\theta$ with $\phi\in C^\infty(M)$ and $\phi>0$, the CR conformal Laplacian $\mathcal{L}_\theta= b_n\Delta_b+R$ satisfies
\begin{equation}\label{conformal Laplacian}
    \mathcal{L}_{\tilde{\theta}}(\phi^{-1}u) =\phi^{-\frac{Q+2}{Q-2}}\mathcal{L}_\theta(u),
\end{equation}
where $R$ is the Tanaka-Webster scalar curvatures and $b_n=\frac{2Q}{Q-2}$. Take $u=\phi$, we have the prescribed curvature equation \eqref{CR Yamabe equation}. Furthermore, for given constant curvature $\tilde{R}$, the existence of \eqref{CR Yamabe equation} is known as \textit{CR Yamabe problem}, which was introduced by Jerison and Lee, see \cite{JL1987, JL1983}.

\subsection{Heisenberg group $\mathbb{H}^n$, complex sphere $\mathbb{S}^{2n+1}$ and the Cayley transform}\ 
%
\textit{Heisenberg group} $\mathbb{H}^n$ is $\mathbb{C}^n \times \mathbb{R}=\{u=(z,t):z=(z_1, \cdots, z_n)=(x_1+\sqrt{-1}y_1,\cdots,x_n+\sqrt{-1}y_n,t)\in \mathbb{C}^n,t\in  \mathbb{R}\}$
with the multiplication law
    $$(z,t)(z',t')=(z+z', t+t'+2Im(z \cdot \overline{z'})),$$
where $z \cdot \overline{z'}=\sum_{j=1}^n z_j \overline{z_j'}.$ A class of natural norm function is given by $|u|=(|z|^4+t^2)^{1/4}$ and the distance between points $u$ and $v$ is defined by $d(u,v)=|v^{-1}u|$. Take a set of basis $Z_j=\frac\partial{\partial z_j}+\sqrt{-1}\bar{z}_j\frac{\partial}{\partial t},\ j=1,2,\cdots,n$, which spans $T_{1,0}$ of $\mathbb{H}^n$. A natural $1$-form is
    $$\theta_0=dt+\sqrt{-1}\sum_{j=1}^n(z_j d\bar{z}_j-\bar{z}_j dz_j),$$
and the corresponding \textit{Levi form} can be specified as
    $$L_{\theta_0}(Z_j,\overline{Z}_k)=\delta_{jk},\quad j,k=1,2,\cdot,n.$$
The \textit{Sublaplacian} are defined with respect to $\{T_j,\ j=1,2,\cdots,n\}$ as
\begin{equation}\label{sub-Laplacian Hn}
    \triangle_H=-\sum_{j=1}^n(Z_j\overline{Z_j}+\overline{Z_j}Z_j).
\end{equation}

The \textit{sphere $\mathbb{S}^{2n+1}$} is $\{\xi=(\xi_1,\cdots,\xi_{n+1}:\ \xi\cdot\bar{\xi}=\sum_{j=1}^{n+1}|\xi_j|^2=1\}\subset \mathbb{C}^{n+1}$ and the subspace $T_{1,0}\subset\mathbb{C}T(\mathbb{S}^{2n+1})$ is spanned by
    $$T_j=\frac\partial{\partial \xi_j}-\overline{\xi_j}\mathcal{R},\quad \mathcal{R}=\sum_{k=1}^{n+1}\xi_k\frac\partial{\partial \xi_k}, \quad j=1,\cdots,n+1.$$
A natural $1$-form is
    $$\theta_S=\sqrt{-1}\sum_{j=1}^{n+1}(\xi_j d\overline{\xi_j}-\overline{\xi_j}d\xi_j)$$
and the \textit{Sublaplacian} is
\begin{equation}\label{sub-Laplacian S}
    \Delta_S=-\sum_{j=1}^{n+1}(T_j\overline{T_j}+\overline{T_j}T_j).
\end{equation}
On the sphere the distance function is defined as $d(\zeta,\eta)^2=2|1-\zeta\cdot\bar{\eta}|$.

\textit{Cayley transform} $\mathcal{C}:\ \mathbb{H}^n\rightarrow \mathbb{S}^{2n+1}\backslash (0,0,\cdots,0,-1)$ and its reverse are defined as
\begin{gather*}
    C(z,t)=\Bigl(\frac{2z}{1+|z|^2+it},\frac{1-|z|^2-it}{1+|z|^2+it}\Bigr),\\
    C^{-1}(\xi) =\Bigl( \frac{\xi_1}{1+\xi_{n+1}}, \cdots, \frac{\xi_n}{1+\xi_{n+1}}, \rm{Im}\frac{1-\xi_{n+1}}{1+\xi_{n+1}}\Bigr),
\end{gather*}
respectively. The Jacobian of the Cayley transform is
    $$J_\mathcal{C}{(z,t)}=\frac{2^{2n+1}}{((1+|z|^2)^2+t^2)^{n+1}}$$
which implies that
\begin{equation}\label{transform integral Hn S}
    \int_{\mathbb{H}^n} f dV_0 =2^{2n}n! \int_{\mathbb{H}^n} f du
    =\int_{\mathbb{S}^{2n+1}} F dV_S =2^{2n+1} n! \int_{\mathbb{S}^{2n+1}} F d\xi,
\end{equation}
where $f=(F\circ \mathcal{C})(2|J_\mathcal{C}|)$, $dV_0=\theta_0\wedge d\theta_0 \wedge\cdots\wedge d\theta_0 $, $dV_S=\theta_S \wedge d\theta_S \wedge\cdots\wedge d\theta_S$, $du=dzdt=dxdydt$ is the Haar measure on $\mathbb{H}^n$ and $d\xi$ is the Euclidean volume element of $\mathbb{S}^{2n+1}$. Moreover, through the Cayley transform, we have the following relations between two distance functions
\begin{equation}\label{formula relations distance}
    d(\zeta,\eta)=d(u,v)\left(\frac 4{(1+|z|^2)^2+t^2}\right)^{1/4}\left(\frac 4{(1+|z'|^2)^2+t'^2}\right)^{1/4},
\end{equation}
where $\zeta=\mathcal{C}(u),\ \eta=\mathcal{C}(v),\ u=(z,t)$ and $v=(z',t')$.

Adopting the above notations, we can rewrite the sharp Hardy-Littlewood-Sobolev inequalities on $\mathbb{H}^n$ and $\mathbb{S}^{2n+1}$ (see Frank and Lieb's result \cite{Frank-Lieb2012}) as
\begin{theorem}[Sharp HLS inequality on $\mathbb{H}^n$]\label{thm HLS Hn}
For $0<\alpha<Q$ and $p=\frac{2Q}{Q+\alpha}$. Then for any $f,g\in L^p(\mathbb{H}^n)$,
\begin{equation}\label{sharp HLS Hn}
    \left|\int_{\mathbb{H}^n} \int_{\mathbb{H}^n} \frac{\overline{f(u)} g(v)}{d(u,v)^{Q-\alpha}} dV_0(u) dV_0(v)\right|\le D_{H} ||f||_{L^{p}(\mathbb{H}^n,dV_0)}||g||_{L^{p}(\mathbb{H}^n,dV_0)}
\end{equation}
where
\begin{equation}\label{sharp Con Hn}
D_{H}:=(2\pi)^{\frac{Q-\alpha}2} \frac{n!\Gamma(\alpha/2)}{\Gamma^2((Q+\alpha)/4)} =\mathcal{Y}_\alpha(\mathbb{H}^n).
\end{equation}
And equality holds if and only if
\begin{equation}\label{extremal Hn}
f(u)=c_1g(u)=c_2H(\delta_r(a^{-1}u),
\end{equation}
for some $c_1, \ c_2\in\mathbb{C}$,\ $r>0$ and $a\in\mathbb{H}^n$ (unless $f\equiv 0$ or $g\equiv 0$). Here $H$ is defined as
\begin{equation}\label{extremal1 Hn}
    H(u)=H(z,t)=((1+|z|^2)^2+t^2)^{-(Q+\alpha)/4}.
\end{equation}
\end{theorem}

\begin{theorem}[Sharp HLS inequality on $\mathbb{S}^{2n+1}$]\label{thm HLS Sn}
For $0<\alpha<Q$ and $p=\frac{2Q}{Q+\alpha}$. Then for any $f,g\in L^p(\mathbb{S}^{2n+1})$,
\begin{equation}\label{sharp HLS Sn}
    \left|\int_{\mathbb{S}^{2n+1}} \int_{\mathbb{S}^{2n+1}} \frac{\overline{f(\zeta)} g(\eta)}{d(\zeta,\eta)^{(Q-\alpha)}} dV_s(\zeta) dV_S(\eta)\right|\le D_{S} \|f\|_{L^{p}(\mathbb{S}^{2n+1},dV_S)} \|g\|_{L^{p}(\mathbb{S}^{2n+1},dV_S)}
\end{equation}
where
\begin{equation}\label{sharp Con Sn}
    D_{S}:=D_H=\mathcal{Y}_\alpha(\mathbb{S}^{2n+1}). 
\end{equation}
And equality holds if and only if
\begin{equation}\label{extremal Sn}
f(\zeta)=c_1|1-\bar\xi\cdot\zeta|^{-(Q+\alpha)/2},\quad g(\zeta)=c_2 |1-\bar\xi\cdot\zeta|^{-(Q+\alpha)/2}
\end{equation}
for some $c_1, \ c_2\in\mathbb{C}$ and some $\xi\in\mathbb{C}^{n+1}$ with $|\xi|<1$ (unless $f\equiv 0$ or $g\equiv 0$).
\end{theorem}

\subsection{Folland-Stein normal coordinates (see \cite{Folland-Stein1974, JL1987})} On some open set $V\subset M$, take a set of \textit{pseudohermitian frame} $\{W_1,\cdots,W_n\}$. Then, $\{W_i,\overline{W_i},T, i=1,\cdots,n\}$ forms a local frame, 
where $T$ is determined by $\theta(T)=1$ and $d\theta(T,X)=0$ for all $X\in\ TM$. As the \textbf{Theorem 4.3} and \textbf{Remark 4.4} of \cite{JL1987}, we can summarize the result of Folland-Stein normal coordinates in the following theorem.
\begin{theorem}[Theorem 4.3 of \cite{JL1987}]\label{thm 4.3 of JL}
There is a neighbourhood $\Omega\subset M\times M$ of the diagonal and a $C^\infty$ mapping $\Theta: \Omega\rightarrow \mathbb{H}^n$ satisfying:

{\rm (a)} $\Theta(\xi,\eta) =-\Theta(\eta,\xi) =\Theta(\eta,\xi)^{-1} \in\mathbb{H}^n$. (In particular, $\Theta(\xi,\xi)=0$.)

{\rm (b)} Denote $\Theta_\xi(\eta)=\Theta(\xi,\eta)$. $\Theta_\xi$ is thus a diffeomorphism of a neighbourhood $\Omega_\xi$ of $\xi$ onto a neighbourhood of the origin in $\mathbb{H}^n$. Denote by $u=(z,t)=\Theta(\xi,\eta)$ the coordinates of $\mathbb{H}^n$. Denote by $O^k,\ k=1,2,\cdots,$ a $C^\infty$ function $f$ of $\xi$ and $y$ such that for each compact set $K\subset\subset V$ there is a constant $C_K$, with $|f(\xi,y)|\leq C_K|y|^k$ (Heisenberg norm) for $\xi\in K$. Then, 
\begin{gather*}
   (\Theta_\xi^{-1})^*\theta=\theta_0+O^1 dt+\sum_{j=1}^{2n}(O^2 dz_j+O^2 d\bar{z}_j),\\
    (\Theta_\xi^{-1})^*(\theta\wedge d\theta^n)=(1+O^1)\theta_0\wedge d\theta_0^n.
\end{gather*}

{\rm (c)} \begin{align*}
    \Theta_{\xi *}W_j &=Z_j +O^1\mathscr{E}(\partial_z) +O^2\mathscr{E}(\partial_t),\quad \Theta_{\xi *}T=\frac\partial{\partial t} +O^1\mathscr{E}(\partial_z,\partial_t),\\
    \Theta_{\xi *}\Delta_b 
    &=\Delta_H +\mathscr{E}(\partial_z) +O^1\mathscr{E}(\partial_t,\partial_z^2) +O^2\mathscr{E}(\partial_z\partial_t) +O^3\mathscr{E}(\partial_t^2),
\end{align*}
in which $O^k\mathscr{E}$ indicates an operator involving linear combinations of the indicated derivatives with coefficients in $O^k$, and we have used $\partial_z$ to denote any of the derivatives $\partial/\partial z_j,\ \partial/\partial \bar{z}_j$. (The uniformity with respect to $\xi$ of bounds on functions in $O^k$ is not stated explicitly \cite{Folland-Stein1974}, but follows immediately from the fact that the coefficients are $C^\infty$.)
\end{theorem}

\begin{theorem}[Remark 4.4 of \cite{JL1987}]
Let $T^\delta(z,t)=(\delta^{-1}z,\delta^{-2}t),\ K\subset\subset V$, and let $r$ be fixed. With the notation of Theorem \ref{thm 4.3 of JL} and $B_r=\{ u\in\mathbb{H}^n:\ |u|\leq r\}$, then $T^\delta\circ\Theta_\xi(\Omega_\xi)\supset B_r$ for sufficiently small $\delta$ and all $\xi\in K$. Moreover, for $\xi\in K$ and $u\in B_r$,
\begin{align*}
    &\bigl((T^\delta\circ\Theta_\xi)^{-1}\bigr)^*\theta =\delta^2(1+\delta O^1)\theta_0,\\
    &\bigl((T^\delta\circ\Theta_\xi)^{-1}\bigr)^*(\theta\wedge d\theta^n) =\delta^{2n+2}(1+\delta O^1) \theta_0\wedge d\theta_0^n,\\
    &(T^\delta\circ\Theta_\xi)_* W_j =\delta^{-1}(Z_j +\delta O^1\mathscr{E}(\partial_z) +\delta^2O^2\mathscr{E}(\partial_t)),\\
    &(T^\delta\circ\Theta_\xi)_*\Delta_b =\delta^{-2}(\Delta_H +\mathscr{E}(\partial_z) +\delta O^1\mathscr{E}(\partial_t,\partial_z^2)\\
    &\hspace{2.0cm} +\delta^2\mathscr{E}(\partial_z\partial_t) +\delta^3 O^3\mathscr{E}(\partial_t^2)).
\end{align*}
(Here $O^k$ may depeng also on $\delta$, but its derivatives are bounded by multiplies of the frame constants, uniformly as $\delta\rightarrow 0$. Recall that $T^\delta_*Z_j=\delta^{-1}Z_j$, and $\bigl((T^\delta)^{-1}\bigr)^*\theta_0=\delta^2\theta_0$.)
\end{theorem}

\subsection{Function spaces (see \cite{Folland-Stein1974, JL1987})}
Let $U$ be a relatively compact open subset of a normal coordinate neighbourhood $\Omega_\xi\subset M$ and take $(W_1,\cdots,W_n)$ be a pseudohermitian frame on $U$. Let $X_j=\text{Re}W_j$ and $X_{j+n}=\text{Im}W_j$ for $j=1,\cdots,n$. Denote $X^\alpha=X_{\alpha_1}\cdots\cdots X_{\alpha_k}$, where $\alpha=(\alpha_1,\cdots,\alpha_k)$, each $\alpha_j$ an integer $1\leq \alpha_j\leq 2n$, and denote $l(\alpha)=k$.

The Folland-Stein space $S_k^p(U)$ is defined as the completion of $C_0^\infty(U)$ with respect to the norm
    $$\|f\|_{S_k^p(U)}=\sup_{l(\alpha)\leq k} \|X^\alpha f\|_{L^p(U)}, \quad \text{with}\ \|X^\alpha f\|^p_{L^p(U)}=\int_U |X^\alpha f|^p dV_\theta.$$

On the other hand, Folland and Stein also defined H\"{o}lder spaces $\Gamma_\beta(U)$ as follows. Take $\rho(\xi,\eta)=|\Theta(\xi,\eta)|$ (Heisenberg norm) as the natural distance function on $U$. Then, for $0<\beta<1$ define
    $$\Gamma_\beta(U)=\{f\in C^0(\overline{U}): |f(\eta)-f(\xi)|\leq C\rho(\xi,\eta)^\beta\}$$
with norm
    $$\|f\|_{\Gamma_\beta(U)} =\sup_{\xi\in U} |f(\xi)| +\sup_{\xi,\eta\in U} \frac{|f(\eta)-f(\xi)|}{\rho(\xi,\eta)^\beta}.$$
For $\beta=1$ define
    $$\Gamma_1(U)=\{f\in C^0(\overline{U}): |f(\eta)+f(\tilde\eta)-2f(\xi)|\leq C\rho(\xi,\eta)\}$$
with norm
    $$\|f\|_{\Gamma_1(U)} =\sup_{\xi\in U} |f(\xi)| +\sup_{\xi,\eta\in U} \frac{|f(\eta)+f(\tilde\eta)-2f(\xi)|} {\rho(\xi,\eta)},$$
where $\tilde\eta=\Theta_\xi^{-1}(-\Theta_\xi(\eta))$.
For $\beta=k+\beta'$ with $k=1,2,\cdots$ and $0<\beta'\leq 1$, define
    $$\Gamma_\beta(U)=\{f\in C^0(\overline{U}):\ X^\alpha f\in \Gamma_{\beta'}(U) \text{ for } l(\alpha)\leq k\}$$
with norm
    $$\|f\|_{\Gamma_\beta(U)}=\begin{cases} \sup\limits_{\xi\in U}|f(\xi)| +\sup\limits_{\substack{\xi,\eta\in U\\ l(\alpha)\leq k}} \frac{|X^\alpha f(\eta)-X^\alpha f(\xi)|}{\rho(\xi,\eta)^{\beta'}}, \quad 0<\beta'<1,\\
    \sup\limits_{\xi\in U}|f(\xi) +\sup\limits_{\substack{\xi,\eta\in U\\ l(\alpha)\leq k}} \frac{|X^\alpha f(\eta)+X^\alpha f(\tilde\eta)-2X^\alpha f(\xi)|}{\rho(\xi,\eta)}, \quad \beta'=1. \end{cases}$$

Fix the local coordinates of $U$ by $u=(z,t)=\Theta_\xi$ for some given point $\xi\in U$. Then, for $0<\beta<1$, the standard H\"{o}lder space $\Lambda_\beta(U)$ is
    $$\Lambda_\beta(U)=\{f\in C^0(\overline{U}): |f(u)-f(v)|\leq C\|u-v\|^\beta\}$$
with norm
    $$\|f\|_{\Lambda_\beta(U)} =\sup_{u\in U} |f(\xi)| +\sup_{u,v\in U} \frac{|f(u)-f(v)|}{\|u-v\|^\beta}.$$
While for $\beta\geq 1$, $\Lambda_\beta(U)$ can be defined similarly. Furthermore, we have
\begin{proposition}[Theorem 20.1 of \cite{Folland-Stein1974} \& Proposition 5.7 of \cite{JL1987}]\label{thm 20.1 of FS}
$\Gamma_\beta\subset \Lambda_{\beta/2}(\text{loc})$ for $0<\beta<\infty$ and there exists some positive constant $C$ such taht $\|f\|_{\Lambda_{\beta/2}(U)}\leq C\|f\|_{\Gamma_\beta(U)}$ for any $f\in C_o^\infty(U)$.
\end{proposition}

Now for a compact strictly pseudoconvex pseudohermitian manifold $M$, choose a finite open covering $U_1,\cdots,U_m$ for which each $U_j$ has the properties of $U$ above. Choose a $C^\infty$ partition of unity $\phi_i$ subordinate to this covering, and define
\begin{gather*}
    S_k^p(M)=\{f\in L^1(M):\ \phi_i f\in S_k^p(U) \text{ for all}  i\},\\
    \Gamma_\beta(M)=\{f\in C^0(M):\ \phi_if\in \Gamma_\beta(U_j) \text{ for all } i\}.
\end{gather*}

Following, for convenience, denote $p_\alpha=\frac{2Q}{Q-\alpha}$ and $q_\alpha=\frac{2Q}{Q+\alpha}$.

\section{Estimation of the sharp constant}\label{Sec estimate}

\begin{proposition}
    $$\mathcal{Y}_\alpha(M)\geq D_H.$$
\end{proposition}
\begin{proof}
Since $(G_\xi^\theta(\eta))^{\frac{Q-\alpha}{Q-2}}\sim \rho(\xi,\eta)^{\alpha-Q}$ as $\rho(\xi,\eta)\rightarrow 0$, then for any small enough $\delta>0$, there exists a neighbourhood $V$ of the diagonal of $M\times M$ such that
\begin{equation}\label{formula compare}
    (1-\delta)\rho(\xi,\eta)^{\alpha-Q}\leq (G_\xi^\theta(\eta))^{\frac{Q-\alpha}{Q-2}} \leq (1+\delta)\rho(\xi,\eta)^{\alpha-Q}.
\end{equation}

Recall  that $f(u)=H(u)$ is an extremal function to the sharp HLS inequality in Theorem \ref{thm HLS Hn}, as well as its conformal equivalent class:
\begin{equation}\label{fe}
f_\epsilon(u)=\epsilon^{-\frac{Q+\alpha}2} H(\delta_{\epsilon^{-1}}(u)),\quad \forall\epsilon>0.
\end{equation}
Thus
\[
\|I_\alpha f\|_{L^{p_\alpha}(\mathbb{H}^n)}=\|I_\alpha f_\epsilon\|_{L^{p_\alpha}(\mathbb{H}^n)},
\quad\|f\|_{L^{q_\alpha}(\mathbb{H}^n)} =\|f_\epsilon\|_{L^{q_\alpha}(\mathbb{H}^n)},
\]
and $f_\epsilon(u)$ satisfies integral equation
\begin{equation}\label{EL-equ-1}
f_\epsilon^\frac{Q-\alpha}{Q+\alpha}(u)=B\int_{\mathbb{H}^n} \frac{f_\epsilon(v)}{|v^{-1}u|^{Q-\alpha}}dV_0(v),
\end{equation}
where $B$ is a positive constant. 

Let $\Sigma_R=\{u=(z,t)\in\mathbb{H}^n:\ |z|<R,\ |t|<R^2\}$ be a cylindrical set, where $R$ is a fixed constant to be determined later, and take a test function $g(u)\in L^{q_\alpha}(\mathbb{H}^n)$ as
 \begin{eqnarray*}
g(u)= \begin{cases}f_\epsilon(u) &\quad u\in \Sigma_{R}(\zeta),\\
 0&\quad u\in\Sigma_{R}^C=\mathbb{H}^n\backslash \Sigma_{R}.
  \end{cases}
  \end{eqnarray*}
Then, 
\begin{eqnarray*}
& &\int_{\mathbb{H}^n} \int_{\mathbb{H}^n} \frac{g(u)g(v)} {|v^{-1}u|^{Q-\alpha}} dV_0(u) dV_0(v)=\int_{\mathbb{H}^n} \int_{\mathbb{H}^n} \frac{ f_\epsilon(u) f_\epsilon(v)}{|v^{-1}u|^{Q-\alpha}} dV_0(u) dV_0(v)\\
& &-2\int_{\mathbb{H}^n} \int_{\Sigma_{R}^C} \frac{ f_\epsilon(u) f_\epsilon(v)}{|v^{-1}u|^{Q-\alpha}} dV_0(u) dV_0(v)
+\int_{\Sigma_{R}^C } \int_{\Sigma_{R}^C } \frac{ f_\epsilon(u) f_\epsilon(v)}{|v^{-1}u|^{Q-\alpha}} dV_0(u) dV_0(v)\\
&&=D_H\|f_\epsilon\|^2_{L^{q_\alpha}(\mathbb{H}^n)}-I_1+I_2,
\end{eqnarray*}
where
\begin{align*}
I_1:=&2\int_{\mathbb{H}^n} \int_{\Sigma_{R}^C(\zeta) } \frac{ f_\epsilon(u) f_\epsilon(v)}{|v^{-1}u|^{Q-\alpha}} dV_0(u) dV_0(v)\\
I_2:=&\int_{\Sigma_{R}^C(\zeta) } \int_{\Sigma_{R}^C(\zeta) } \frac{ f_\epsilon(u) f_\epsilon(v)}{|v^{-1}u|^{Q-\alpha}} dV_0(u) dV_0(v).
\end{align*}
With \eqref{EL-equ-1}, we have
\begin{align*}
I_1 =
C\int_{\Sigma_{R}^C} f_\epsilon^{\frac{2Q}{Q+\alpha}}(u) dV_0(u) =O(\frac{R}\epsilon)^{-Q},  \quad\quad\text{as}\quad\epsilon\to 0.
\end{align*}
For $I_2$, by HLS inequality \eqref{sharp HLS Hn}, we have
\begin{eqnarray*}
I_2&\le&D_H\|f_\epsilon\|^2_{L^{q_\alpha}(\Sigma_R^C(\zeta))}= O(\frac{R}\epsilon)^{-Q-\alpha}\quad\quad\text{as}\quad\epsilon\to 0.
\end{eqnarray*}

Hence, for small enough $\epsilon$, we have
\begin{eqnarray}\label{estimate 1}
    \frac{\int_{\Omega} \int_{\Omega} g(u)g(v)|v^{-1}u|^{\alpha-Q} dV_0(u) dV_0(v)}{\|g\|^2_{L^{q_\alpha}(\Omega)}}\geq D_H-O(\frac{R}\epsilon)^{-Q}.
\end{eqnarray}

For any given point $\xi\in M$, there exists a neighbourhood $V_\xi\subset V$ such that 
Theorem \ref{thm 4.3 of JL} hold. So, choose $R$ small enough such that $\Sigma_R\subset\Theta_\xi(V_\xi)$ 
and $(\Theta_\xi^{-1})^*(dV_\theta)=(1+O^1)dV_0$. Let
    $$\Phi(\eta)=\begin{cases}f_\epsilon(\Theta_\xi(\eta)), &\quad \text{in}\quad E_\xi(\Sigma_R),\\
    0,& \quad\text{in}\quad M\setminus E_\xi(\Sigma_R).\end{cases}$$
Then,
\begin{align*}
    &\int_M |\Phi(\eta)|^{2Q/(Q+\alpha)} dV_\theta(\eta) = \int_{\Sigma_R} |f_\epsilon(u)|^{2Q/(Q+\alpha)} (1+O^1) dV_0,\\ 
    &\iint_{M\times M} \Phi(\eta)\Phi(\zeta) (G^\zeta_\xi(\eta))^{\frac{Q-\alpha}{Q-2}} dV_\theta(\eta) dV_\theta(\zeta)\nonumber\\
    \geq &(1-\delta)\iint_{M\times M} \Phi(\eta)\Phi(\zeta) \rho(\eta,\zeta)^{\alpha-Q} dV_\theta(\eta) dV_\theta(\zeta)\nonumber\\
    =& (1-\delta) \iint_{\Sigma_R\times\Sigma_R} f_\epsilon(u) f_\epsilon(v) |u^{-1}v|^{\alpha-Q} (1+O^1)^2 dV_0(u)dV_0(v), 
\end{align*}
and
\begin{align*}
    \mathcal{Y}_\alpha(M) &\geq \frac{\iint_{M\times M} \Phi(\eta)\Phi(\zeta) (G^\zeta_\xi(\eta))^{\frac{Q-\alpha}{Q-2}} dV_\theta(\eta) dV_\theta(\zeta)}{\|\Phi\|_{L^{2Q/(Q+\alpha)}}^2(M)}\\
    &\geq \frac{(1-\delta)\iint_{\Sigma_R\times\Sigma_R} f_\epsilon(u) f_\epsilon(v) |u^{-1}v|^{\alpha-Q} (1+O^1)^2 dV_0(u) dV_0(v)} {(1+O^1)^{(Q+\alpha)/Q}  \|f_\epsilon\|_{L^{2Q/(Q+\alpha)}}^2(\Sigma_R)}\\
    &\geq (1-\delta)(1+O^1)^{2-(Q+\alpha)/Q}(D_H-O(\frac R\epsilon)^{-Q})+O(\epsilon^{Q-\alpha}).
\end{align*}
sending $R$ to $0$ and then letting $\delta, \epsilon$ approach to zero, we obtain the estimate.
\end{proof}

\section{Subcritical HLS inequalities on compact CR manifold}\label{Sec sub}

\subsection{Subcritical HLS inequalities and their extremal function}

\begin{proposition}[Young's inequality]\label{pro Young ineq}
Let $X$ and $Y$ are measurable spaces, and let the kernel function $K: X\times Y\rightarrow R$ be a measuralble function satisfying
    $$\int_X |K(x,y)|^r dx\leq C^r \quad\text{and}\quad \int_Y |K(x,y)|^r dy\leq C^r,$$
where $C$ is some positive constant and $r\geq 1$. Then, for any $f\in L^p(Y)$ with $1-1/r\leq 1/p\leq 1$, the integral operator
    $$Af(x)=\int_Y K(x,y)f(y)dy$$
satisfies $\|Af\|_{L^q(X)}\leq C\|f\|_{L^p(X)}$, where $1\leq p\leq q$ and $\frac 1q=\frac 1p+\frac 1r-1$.
\end{proposition}
\begin{proof}
For the case $r=1$, the result reduces to the case of Lemma 15.2 of \cite{Folland-Stein1974}.

If $r>1$, then $1\leq p\leq \frac r{r-1}$. When $p=1$, by Minkowski's inequality,
    $$\left(\int_X |Af(x)|^r dx\right)^{1/r}\leq \int_Y\left(\int_X |K(x,y) f(y)|^r dx\right)^{1/r} dx \leq C\|f\|_{L^1(Y)}.$$
While for $p>1$, noting $(1-\frac 1p)+(1-\frac 1r)+\frac 1q=1$, we have
\begin{align*}
    &|Af(x)|\leq \int_Y |K(x,y)|^{r(1-1/p)} |K(x,y)|^{r/q} |f(y)|^{p/q} |f(y)|^{p(1-1/r)}dy\\
    &\leq \left(\int_Y |K(x,y)|^r dy\right)^{1-1/p} \left(\int_Y |K(x,y)|^r|f(y)|^p dy\right)^{1/q} \left(\int_Y |f(y)|^p dy\right)^{1-1/r},
\end{align*}
and then
\begin{align*}
    &\left(\int_X |Af(x)|^q dx\right)^{1/q}\\
    \leq &\left(\int_Y |K(x,y)|^r dy\right)^{1-1/p} \|f\|_{L^p(Y)}^{p(1-1/r)} \left(\int_X\int_Y|K(x,y)|^r|f(y)|^p dy dx\right)^{1/q}\\
    =& \left(\int_Y |K(x,y)|^r dy\right)^{1/r} \|f\|_{L^p(Y)} \leq C\|f\|_{L^p(Y)}.
\end{align*}
\end{proof}

Since $(G_\xi^\theta(\eta))^{\frac{Q-\alpha}{Q-2}}\sim \rho(\xi,\eta)^{\alpha-Q}$ as $\rho(\xi,\eta)\rightarrow 0$ and $(G_\xi^\theta(\eta))^{\frac{Q-\alpha}{Q-2}} \in C^\infty(M\times M \backslash \{\xi=\eta\})$, then $(G_\xi^\theta(\eta))^{\frac{Q-\alpha}{Q-2}}\leq C\rho(\xi,\eta)^{\alpha-Q}$ 
and $(G_\xi^\theta(\eta))^{\frac{Q-\alpha}{Q-2}}$ satisfies
    $$\int_M |K(\xi,\eta)|^r dV_\theta(\xi)\leq C  \quad\text{and}\quad \int_M |K(\xi,\eta)|^r dV_\theta(\eta)\leq C, \quad \forall 1\leq r <\frac{Q}{Q-\alpha}.$$
So, we take $(G_\xi^\theta(\eta))^{\frac{Q-\alpha}{Q-2}}$ as the kernel of Proposition \ref{pro Young ineq} and have the following subcritical HLS inequalities.
\begin{proposition}\label{pro sub HLS}
There exists a positive constant $C$ such that for any $f\in L^p(M)$ with $1<p<\frac{Q}\alpha$, one has
\begin{equation}\label{roughly sub HLS}
    \|Af(\xi)\|_{L^q(M)}\leq C\|f\|_{L^p(M)},
\end{equation}
where $q>1$ and $\frac 1q>\frac 1p-\frac\alpha{Q}$. Moreover, operator $A$ is compact for any $q$ satisfying $q>1$ and $\frac 1q>\frac 1p-\frac\alpha Q$, namely, for any bounded sequence $\{f_j\}_{j=1}^{+\infty}\subset L^p(M)$, there exists a subsequence of $\{Af_j\}_{j=1}^{+\infty}$ which converges in $L^q(M)$.
\end{proposition}
\begin{proof}
Obviously, it is sufficient to prove the compactness of the operator $A$ with kernel $K(\xi,\eta)=\rho(\xi,\eta)^{\frac{Q-\alpha}{Q-2}}$.

Since the sequence $\{f_j\}$ is bounded in $L^p(M)$, then there exists a subsequence (still denoted by $\{f_j\}$) and a function $f\in L^p(M)$ such that
\begin{equation}\label{comp 1}
    f_j\rightharpoonup f \quad\text{weakly in}\quad L^p(M).
\end{equation}

For $K(\xi,\eta)$, we decompose the integral operator as
    $$\int_M K(\xi,\eta) f_j(\eta) dV_\theta(\eta) \triangleq\int_M K^s(\xi,\eta) f_j(\eta) dV_\theta(\eta)+\int_M K_s(\xi,\eta) f_j(\eta) dV_\theta(\eta),$$
where
    $$K^s(\xi,\eta)=\begin{cases}K(\xi,\eta), & \text{if}\quad \rho(\xi,\eta)>s,\\ 0 & \text{others}\end{cases}$$
and $K_s(\xi,\eta)=K(\xi,\eta)-K^s(\xi,\eta)$, $s>0$ will be chosen later. Noting that, for any $\xi$,
$\int_M |K^s(\xi,\eta)|^{p/(p-1)} dV_\theta(\eta)\leq C_0$, where $C_0$ is independent on $\xi$, we know that $\int_M K^s(\xi,\eta) f_j(\eta) dV_\theta(\eta)$ converges pointwise to $\int_M K^s(\xi,\eta) f(\eta) dV_\theta(\eta)$ and $\int_M K^s(\xi,\eta) f_j(\eta) dV_\theta(\eta)$ is bounded uniformly. So, it is deduced by the dominated convergence that
\begin{equation}\label{comp 3}
    \int_M K^s(\xi,\eta) f_j(\eta) dV_\eta\rightarrow \int_M K^s(\xi,\eta) f(\eta) dV_\eta, \quad\text{in}\quad L^q(M).
\end{equation}

Next, we analyze the convergence of $\int_M K_s(\xi,\eta) f_j(\eta) dV_\eta$ by Young's inequality (Proposition \ref{pro Young ineq}). Take $r$ satisfying $\frac 1r=\frac 1q-\frac 1p+1$. Then, $r<\frac{Q}{Q-\alpha}$ and
    $$\left(\int_M |K_s(\xi,\eta)|^r dv_\eta\right)^{1/r} \leq Cs^\beta$$
with $\beta=Q(\frac 1r-\frac{Q-\alpha}Q)$. By the Young's inequality, we have
\begin{align}\label{comp 4}
    &\|\int_M K_s(\xi,\eta) (f_j(\eta)-f(\eta)) dV_\eta\|_{L^q(M)}\nonumber\\
    \leq &\|f_j-f\|_{L^p(M)}\left(\int_M |K_s(\xi,\eta)|^r dV_\eta\right)^{1/r}\leq Cs^\beta.
\end{align}
By now, through choosing first $s$ small and then $j$ large, we deduce by \eqref{comp 3} and \eqref{comp 4} that
    $$\int_M K(\xi,\eta) f_j(\eta) dV_\eta\rightarrow \int_M K(\xi,\eta) f(\eta) dV_\eta, \quad\text{in}\quad L^q(M).$$
So, we complete the proof.
\end{proof}

\medskip
Define the extremal problem as
\begin{align}\label{extremal sub}
    D_{M,p,q} :=& \sup_{f\in L^{p}(M)\setminus\{0\}} \frac{\|Af\|_{L^q(M)}}{\|f\|_{L^p(M)}}\nonumber\\
    =&\sup_{\substack{f\in L^{p}(M)\setminus\{0\}\\ g\in L^{q'}(M)\setminus\{0\}}} \frac{|\int_M\int_M g(\xi) \bigl(G^\theta_\xi(\eta)\bigr)^{\frac{Q-\alpha}{Q-2}} f(\eta) dV_\theta(\eta) dV_\theta(\xi)|}{\|f\|_{L^p(M)} \|g\|_{L^{q'}(M)}},
\end{align}
where $1<p<\frac Q\alpha$, $1>\frac 1q>\frac 1p-\frac\alpha Q$ and $q'=\frac q{q-1}$. Obviously, we know that $D_{M,p,q}<+\infty$ because of Proposition \ref{pro sub HLS}. Moreover, we have
\begin{theorem}\label{thm extremal sub}
$D_{M,p,q}$ can be attained, namely, there exists a nonnegative function $f\in L^p(M)$ satisfying $\|f\|_{L^p(M)}=1$ and $D_{M,p,q}=\|Af\|_{L^q(M)}.$
\end{theorem}
\begin{proof}
Without loss of generality, choose a nonnegative maximizing sequence $\{f_j\}_{j=1}^{+\infty}\subset L^p(M)$ such that $\|f_j\|_{L^p(M)}=1$ and
    $$\lim_{j\rightarrow +\infty} \|Af_j\|_{L^q(M)}=D_{M,p,q}.$$
Combining the boundedness of the sequence $\{f_j\}$ in $L^p(M)$ and the compactness result (see Proposition \ref{pro sub HLS}), there exists a subsequence (still denoted by $\{f_j\}$) and a function $f\in L^p(M)$ such that
\begin{gather*}
    f_j\rightharpoonup f \quad\text{weakly in}\quad L^p(M),\\
    Af_j\rightarrow Af \quad\text{strongly in}\quad L^q(M).
\end{gather*}
Thus $\|f\|_{L^p(M)}\leq\liminf_{j\rightarrow +\infty}\|f_j\|_{L^p(M)}$ and
    $$D_{M,p,q}=\lim_{j\rightarrow +\infty} \frac{\|Af_j\|_{L^q(M)}}{\|f_j\|_{L^p(M)}} \leq \frac{\|Af\|_{L^q(M)}}{\|f\|_{L^p(M)}}.$$
So, $f$ is a maximizer.
\end{proof}

\subsection{Subcritical HLS inequalities for the diagonal case}

To discuss the extremal problem related to $\mathcal{Y}_\alpha(M)$, we need only to consider the diagonal case. Hence, corresponding to Theorem \ref{thm extremal sub}, we have
\begin{theorem}\label{thm extremal sub diagonal}
For any $f\in L^p(M)$ with $\frac{2Q}{Q+\alpha}<p<2$, it holds
\begin{equation*}
    |\int_M\int_M f(\xi)f(\eta) \bigl(G^\theta_\xi(\eta)\bigr)^{\frac{Q-\alpha}{Q-2}} dV_\theta(\xi) dV_\theta(\eta)|\leq D_{M,p}\|f\|_{L^p(M)}^2,
\end{equation*}
where the sharp constant
    $$D_{M,p}:=\sup_{f\in L^{p}(M)\setminus\{0\}} \frac{|\int_M\int_M f(\xi) f(\eta) \bigl(G^\theta_\xi(\eta)\bigr)^{\frac{Q-\alpha}{Q-2}} dV_\theta(\xi) dV_\theta(\eta)|}{\|f\|_{L^p(M)}^2}$$
can be attained by some nonnegative function $f_p\in L^p(M)$ satisfying $\|f_p\|_{L^p(M)}=1$ and
    $$D_{M,p}= \int_M\int_M f(\xi) f(\eta) \bigl(G^\theta_\xi(\eta)\bigr)^{\frac{Q-\alpha}{Q-2}} dV_\theta(\xi) dV_\theta(\eta).$$
\end{theorem}

\begin{remark}
A direct computation deduces that the extremal function $f_p$ satisfies the Euler-Lagrange equation
\begin{equation}\label{extremal equ2 sub}
    2D_{M,p} f^{p-1}(\xi)=\int_M f(\eta) \bigl(G^\theta_\xi(\eta)\bigr)^{\frac{Q-\alpha}{Q-2}} dV_\theta(\eta) +\int_M f(\eta) \bigl(G^\theta_\eta(\xi)\bigr)^{\frac{Q-\alpha}{Q-2}} dV_\theta(\eta).
\end{equation}
Denoted by $g(\xi)=f^{p-1}(\xi)$. Then \eqref{extremal equ2 sub} reduces to
\begin{equation}\label{extremal equ1 sub}
   2D_{M,p} g(\xi)=\int_M g(\eta)^{q-1} \bigl(G^\theta_\xi(\eta)\bigr)^{\frac{Q-\alpha}{Q-2}} dV_\theta(\eta) +\int_M g(\eta)^{q-1} \bigl(G^\theta_\eta(\xi)\bigr)^{\frac{Q-\alpha}{Q-2}} dV_\theta(\eta),
\end{equation}
where $q=\frac p{p-1}$ is the conjugate exponent of $p$.
\end{remark}

By a classical routine, we have the following regularity result.

\begin{proposition}[$\Gamma_\alpha$ regularity]\label{pro regularity alpha}
If $g(\xi)\in L^p(M)$ satisfies \eqref{extremal equ1 sub}, then $g\in\Gamma_\alpha(M)$.
\end{proposition}

The proof can be completed by the following two Lemmas.

\begin{lemma}\label{lem regularity 2}
If $f\in L^\infty(M)$, then $Af\in \Gamma_\alpha(M)$.
\end{lemma}
\begin{proof}
Because of the compactness of $M$, it is sufficient to prove that, for any $\xi\in M$, Lemma \ref{lem regularity 2} holds on the neighbourhood $V_\xi$. Hence, without loss of generality, we restrict variable $\xi$ on a neighbourhood $V_{\xi_0}$ for some point $\xi_0\in M$.

Using the Folland-Stein normal coordinates, we can complete the proof by a similar process of the second part of Lemma 4.3 of \cite{Han-arxiv}. For concise, we omit the details.
%
\end{proof}

\begin{lemma}\label{lem regularity 1}
If $g(\xi)\in L^p(M)$, $2<p<\frac{2Q}{Q-\alpha}$, satisfies \eqref{extremal equ1 sub}, then $g\in L^\infty(M)$.
\end{lemma}
The proof is a repetition line by line of the first part of Lemma 4.3 of \cite{Han-arxiv}. So, it is omitted.

\section{Sharp HLS inequalities on compact CR manifold}\label{Sec existence of criteria}

\begin{theorem}\label{thm sharp HLS}
If $\mathcal{Y}_\alpha(M)>D_H$, then $\mathcal{Y}_\alpha(M)$ is attained by some function $f\in \Gamma_\alpha(M)$.
\end{theorem}

Following, 
we will investigate the limitation of the sequence of solutions $\{f_p\}\subset\Gamma_\alpha(M)$ of \eqref{extremal equ2 sub}, and then complete the proof of Theorem \ref{thm sharp HLS} by compactness.

First, it is routine to prove
\begin{lemma}\label{lem limit sharp constant}
$D_{M,p}\rightarrow \mathcal{Y}_\alpha(M)$ as $p\rightarrow \left(\frac{2Q}{Q+\alpha}\right)^{+}$.
\end{lemma}

\begin{proposition}\label{pro limit bound}
If $D_{M,p}\geq D_H+\epsilon$, then the positive solutions $\{f_p\}_{\frac{2Q}{Q+\alpha}<p<2}$ is uniformly bounded in $\Gamma_\alpha(M)$.
\end{proposition}
\begin{proof}
In view of the proof of Proposition \ref{pro regularity alpha}, it is sufficient to prove that $\{f_p\}_{\frac{2Q}{Q+\alpha}<p<2}$ is uniformly bounded in $L^\infty(M)$. Following, we will prove it by contradiction. Suppose not. Then $f_p(\xi_p)\rightarrow +\infty$ as $p\rightarrow \left(\frac{2Q}{Q+\alpha}\right)^+$, where $f_p(\xi_p)=\max_{\xi\in M}f_p(\xi)$.

Let $\Theta_{\xi_p}$ be normal coordinates. We can assume that there is a fixed neighbourhood $U=B_r(0)$ of the origin in $\mathbb{H}^n$ contained in the image of $\Theta_{\xi_p}$ for all $p$, and for each $p$ we will use $\Theta_{\xi_p}$ to identify $U$ with a neighbourhood of $\xi_p$ with coordinates $(z,t)=\Theta_{\xi_p}$.

For any $u=(z,t)=\Theta_{\xi_p}(\xi)\in U$, we have
\begin{align}\label{blow-up 1}
    & 2D_{M,p} f_p(\Theta_{\xi_p}^{-1}(u))^{p-1}:= 2D_{M,P}f_p(\xi)^{p-1}\nonumber\\
    =&\int_M f_p(\eta) \bigl(G^\theta_\xi(\eta)\bigr)^{\frac{Q-\alpha}{Q-2}} dV_\theta(\eta) +\int_M f_p(\eta) \bigl(G^\theta_\eta(\xi)\bigr)^{\frac{Q-\alpha}{Q-2}} dV_\theta(\eta) \nonumber\\
    =& 2\int_{\Theta_{\xi_p}^{-1}(U)} f_p(\eta)\rho(\xi,\eta)^{\alpha-Q} dV_\theta(\eta)\nonumber\\
    &+\int_{\Theta_{\xi_p}^{-1}(U)} f_p(\eta) E(\xi,\eta) dV_\theta(\eta)+\int_{\Theta_{\xi_p}^{-1}(U)} f_p(\eta) E(\eta,\xi) dV_\theta(\eta)\nonumber\\ &+\int_{M\setminus\Theta_{\xi_p}^{-1}(U)} f_p(\eta) \bigl(G^\theta_\xi(\eta)\bigr)^{\frac{Q-\alpha}{Q-2}} dV_\theta(\eta) +\int_M f_p(\eta) \bigl(G^\theta_\eta(\xi)\bigr)^{\frac{Q-\alpha}{Q-2}} dV_\theta(\eta)\nonumber\\
    =& 2\int_{U} f_p(\Theta_{\xi_p}^{-1}(v)) \rho(u,v)^{\alpha-Q} (1+O^1) dV_0(v) +I+II\nonumber\\
    = & 2(1+O^1) \int_{U} f_p(\Theta_{\xi_p}^{-1}(v)) \rho(u,v)^{\alpha-Q} dV_0(v) +I+II,
\end{align}
where $O^1\rightarrow 0$ as $r\rightarrow 0$,
\begin{gather*}
    I:=\int_{\Theta_{\xi_p}^{-1}(U)} f_p(\eta) E(\xi,\eta) dV_\theta(\eta)+\int_{\Theta_{\xi_p}^{-1}(U)} f_p(\eta) E(\eta,\xi) dV_\theta(\eta),\\
    II:=\int_{M\setminus\Theta_{\xi_p}^{-1}(U)} f_p(\eta) \bigl(G^\theta_\xi(\eta)\bigr)^{\frac{Q-\alpha}{Q-2}} dV_\theta(\eta) +\int_M f_p(\eta) \bigl(G^\theta_\eta(\xi)\bigr)^{\frac{Q-\alpha}{Q-2}} dV_\theta(\eta)\\
\intertext{and}
    E(\xi,\eta)= \bigl(G^\theta_\eta(\xi)\bigr)^{\frac{Q-\alpha}{Q-2}} -\rho(\xi,\eta)^{\alpha-Q}\leq C\rho(\xi,\eta)^{\alpha-2}
\end{gather*}
for any $\xi,\eta\in \Theta_{\xi_p}^{-1}(U)$.

Take $\mu_p=f_p^{-\frac{2-p}{\alpha}}(\xi_p)$ and $g_p(u)=\mu_p^{\frac{\alpha}{2-p}} f_p(\Theta_{\xi_p}^{-1}(\delta_{\mu_p}(u))),\ u\in B_{r/{\mu_p}}(0)$ with $\delta_{\mu_p}(u)=(\mu_p z,\mu_p^2 t)$. Then, $g_p(u)$ satisfies 
\begin{equation}\label{blow-up 2}
    2D_{M,p} g_p(u)^{p-1}= 2(1+O^1)\int_{B_{r/{\mu_p}}(0)} g_p(v) \rho(u,v)^{\alpha-Q} dV_0(v)+I'+II',
\end{equation}
and $g_p(0)=1,\ g_p(u)\in (0,1]$, where
\begin{gather*}
    I'=\mu_p^{\frac{\alpha(p-1)}{2-p}}\times I \quad\text{and}\quad II'=\mu_p^{\frac{\alpha(p-1)}{2-p}}\times II \quad\text{with}\quad \xi=\Theta_{\xi_p}^{-1}(\delta_{\mu_p}(u)).
\end{gather*}
Since
\begin{align*}
  I'\leq &\mu_p^{\frac{\alpha(p-1)}{2-p}} \cdot C\int_U f_p(\Theta_{\xi_p}^{-1}(v))\rho(\delta_{\mu_p}(u),v)^{\alpha-2} dV_0(v)\\
  =& C\mu_p^{Q-2}\int_{B_{r/\mu_p}(0)} g_p(v) \rho(u,v)^{\alpha-2} dV_0(v)\\
  \leq & C\mu_p^{Q-2} \|g_p\|_{L^p(\mathbb{H}^n)}
  =  C\mu_p^{Q-2}\mu^{\frac{\alpha}{2-p}-\frac Qp}\|f_p\|_{L^p}
\end{align*}
and $\frac{\alpha}{2-p}-\frac Qp>0$, we have $I'\rightarrow 0$ as $p\rightarrow\left(\frac{2Q}{Q+\alpha}\right)^+$. Now, we assume further that $\xi\in \Theta_{\xi_p}^{-1}(B_{r/2}(0))$. Then $\bigl(G^\theta_\xi(\eta)\bigr)^{\frac{Q-\alpha}{Q-2}}$ is uniformly bounded on $M\setminus\Theta_{\xi_p}^{-1}(U)$. Hence, on $\Theta_{\xi_p}^{-1}(B_{r/2}(0))$, $II'\rightrightarrows 0$ as $p\rightarrow\left(\frac{2Q}{Q+\alpha}\right)^+$.

On the other hand, for large $R>0$ and $u\in B_{R/2}(0)$, we have
\begin{align}\label{blow-up 3}
    &\int_{B_{r/{\mu_p}}(0)\setminus B_R(0)} g_p(v) \rho(u,v)^{\alpha-Q} dV_0(v)\nonumber\\
    =&\int_{U\setminus B_{\mu_pR}(0)} \mu_p^{\frac{\alpha}{2-p}} f_p(\Theta_{\xi_p}^{-1}(v)) \rho(u,\delta_{\mu_p^{-1}}(v))^{\alpha-Q} \mu_p^{-Q} dV_0(v)\nonumber\\
    =&\mu_p^{\frac{\alpha}{2-p}-\alpha} \int_{U\setminus B_{\mu_pR}(0)}  f_p(\Theta_{\xi_p}^{-1}(v)) \rho(\delta_{\mu_p}(u),v)^{\alpha-Q} dV_0(v)\nonumber\\
    \leq & C\mu_p^{\frac{\alpha}{2-p}-\alpha} \|f_p\|_{L^p(M)}\left(\int_{U\setminus B_{\mu_pR}(0)} \rho(\delta_{\mu_p}(u),v)^{ \frac{(\alpha-Q) p}{p-1}} dV_0(v)\right)^{\frac{p-1}p}\nonumber\\
    \leq &C \mu_p^{\frac{\alpha}{2-p}-\alpha} (\mu_p R)^{\alpha-\frac Qp}\rightarrow 0,
\end{align}
as $p\rightarrow \left(\frac{2Q}{Q+\alpha}\right)^+$ and $R\rightarrow+\infty$.

So, as $p\rightarrow\left(\frac{2Q}{Q+\alpha}\right)^+$, we have $B_{r/{\mu_p}}(0)\rightarrow \mathbb{H}^n$ and $g_p(u)\rightarrow g(u)$ pointwise in $\mathbb{H}^n$, where $g(u)$ satisfies
\begin{equation}\label{blow-up 4}
    2D_{M,p} g(u)^{\frac{Q-\alpha}{Q+\alpha}}\leq  2(1+O^1)\int_{\mathbb{H}^n} g(v) \rho(u,v)^{\alpha-Q} dV_0(v),\quad g(0)=1.
\end{equation}
Also, a direct computation yields
\begin{align*}
    1> & \int_{\Theta_{\xi_p}^{-1}(U)} f_p^p(\eta) dV_\theta(\eta) =\mu_p^{Q-\frac{p\alpha}{2-p}} (1+O^1) \int_{B_{r/{\mu_p}}(0)} g_p^p(u) dV_0(u)\\
    \geq &(1+O^1) \int_{B_{r/{\mu_p}}(0)} g_p^p(u) dV_0(u).
\end{align*}
Thus, $\int_{B_{r/{\mu_p}}(0)} g_p^p(u) dV_0(u)\leq 1$ for sufficient small $r$.
Combining this with \eqref{blow-up 4}, we have
\begin{align*}
    D_H+\epsilon\leq & D_{M,p}\leq \frac{(1+O^1)\int_{\mathbb{H}^n}\int_{\mathbb{H}^n} \frac{g(u)g(v)}{\rho(u,v)^{Q-\alpha}} dV_0(v)\ dV_0(u)} {\int_{\mathbb{H}^n} g(u)^{2Q/(Q+\alpha)} dV_0(u)}\\
    \leq & \frac{(1+O^1)\int_{\mathbb{H}^n}\int_{\mathbb{H}^n} \frac{g(u)g(v)}{\rho(u,v)^{Q-\alpha}} dV_0(v)\ dV_0(u)} {\|g(z,t)\|_{L^{2Q/(Q+\alpha)} (\mathbb{H}^n)}^2} \leq (1+O^1)D_H,
\end{align*}
which deduces a contradiction for sufficient small $r$.
\end{proof}
\medskip

{\bf Proof of Theorem \ref{thm sharp HLS}.}\ Let $f_p$ be solutions to \eqref{extremal equ2 sub} for $p\in(\frac{2Q}{Q+\alpha},2)$, which are also the maximal functions to $D_{m,p}$. Then $\|f_p\|_{\Gamma^\alpha(M)}\leq C$ by Proposition \ref{pro limit bound}, which deduce that there exists a subsequence of $\{f_p\}$ (still denoted by $\{f_p\}$)  and a function $f\in\Gamma^\alpha(M)$ such that $f_p\rightarrow f$ as $p\rightarrow\left(\frac{2Q}{Q+\alpha}\right)^+$ in $\Gamma^\alpha(M)$. Hence, $D_M$ can be attained by $f$.

%

\subsection{Discussion of the condition of strict inequality}

By now, we want to know whether the strict inequality of Theorem \ref{thm sharp HLS} holds for every CR compact manifold $M$ with $\mathcal{Y}(M)>0$, which is not CR conformally to complex sphere $(\mathbb{S}^{2n+1},\ \theta_S)$. Following, we give a partial answer.
\begin{proposition}
For a given compact locally CR conformal flat manifold $M$ with $\mathcal{Y}(M)>0$, if there exists some point $\xi_0\in M$ such that $A(\xi_0) =w(\xi_0,\xi_0)>0$ (see the definition of $A(\xi)$ in \eqref{Green}), then $\mathcal{Y}_\alpha(M)>\mathcal{Y}_\alpha(\mathbb{S}^{2n+1})$.
\end{proposition}
\begin{proof}
By the assumption of $w(\xi,\eta)\in C^1(M\times M)$ and $A(\xi_0)>0$, there exists a small neighbourhood $V_{\xi_0}\in M$ and some positive constant $A_0$ such that $A(\xi)>A_0>0$. Moreover, on $V_{\xi_0}\times V_{\xi_0}$, there exists some positive constant $C$ such that
    $$(G^\theta_\xi(\eta))^{\frac{Q-\alpha}{Q-2}}\geq \rho(\xi,\eta)^{\alpha-Q}+C\rho^{\alpha-2}.$$
Then by a similar argument with \textbf{Proposition 2.9} of \cite{Han-Zhu2016}, we can complete the proof.
\end{proof}

\begin{remark}
In \cite{CMY2017}, Cheng, Malchiodi and Yang have proved a class of positive mass theorem in three dimensional CR geometry. Namely, under the assumptions of $\mathcal{Y}(M)>0$ and the non-negativity of the CR Paneitz operator, they proved that, if $M$ is not CR equivalent to $S^3$ (endowed with its standard CR structure), then $A(\xi)>0$. Therefore, our main result holds for three dimensional compact CR manifold without boundary. While for other cases, it is still an open question as far as we know.
\end{remark}

\begin{appendix}

\section{CR conformality}\label{Sec CR conformality}

For completeness, we give the CR conformality of operator $I_{M,\theta,\alpha}$ defined in \eqref{conformal operator}.

\begin{proposition}
Let $\tilde{\theta}=\phi^{\frac 4{Q-2}}\theta$. Then
\begin{equation}\label{Trans of Green function}
    G_\xi^{\tilde{\theta}}(\eta) =\phi^{-1}(\xi)\phi^{-1}(\eta)G_\xi^\theta(\eta).
\end{equation}
\end{proposition}
\begin{proof}
It is sufficient to prove that, for any $u\in C^2(M)$,
\begin{equation}\label{Green 1}
    \int_M G_\xi^{\tilde{\theta}}(\eta) u(\eta)\ dV_{\tilde{\theta}} =\int_M \phi^{-1}(\xi) \phi^{-1}(\eta)G_\xi^\theta(\eta) u(\eta)\ dV_{\tilde{\theta}}.
\end{equation}
In fact, by \eqref{conformal Laplacian}, we have
\begin{align*}
    &\mathcal{L}_{\tilde{\theta}} \Bigl[ \int_M \phi^{-1}(\xi) \phi^{-1}(\eta)G_\xi^\theta(\eta) u(\eta)\ dV_{\tilde{\theta}} \Bigr]\\
    =&\phi^{-\frac{Q+2}{Q-2}}(\xi) \mathcal{L}_\theta \Bigl[ \int_M \phi^{-1}(\eta)G_\xi^\theta(\eta) u(\eta)\ dV_{\tilde{\theta}} \Bigr]\\
    =& \phi^{-\frac{Q+2}{Q-2}}(\xi) \mathcal{L}_\theta \Bigl[ \int_M \phi^{-1}(\eta)G_\xi^\theta(\eta) u(\eta) \phi(\eta)^{\frac{2Q}{Q-2}}(\eta)\ dV_\theta \Bigr]=u(\xi).
\end{align*}
So, we complete the proof by the property of Green function.
\end{proof}


\begin{proposition}
For positive constant $\alpha\neq Q$, we have
\begin{equation}\label{conformal operator property 1}
    I_{M,\tilde{\theta},\alpha}(u) =\phi^{-\frac{Q-\alpha}{Q-2}} I_{M,\theta,\alpha}(\phi^{\frac{Q+\alpha}{Q-2}}u).
\end{equation}
\end{proposition}
\begin{proof}
\begin{align*}
    I_{M,\tilde{\theta},\alpha}(u) &=\int_M [G_\xi^{\tilde{\theta}}(\eta)]^{\frac{Q-\alpha}{Q-2}} u(\eta)\ dV_{\tilde{\theta}}\\
    &=\phi^{-\frac{Q-\alpha}{Q-2}}\int_M [G_\xi^{\theta}(\eta)]^{\frac{Q-\alpha}{Q-2}} \phi(\eta)^{\frac{Q+\alpha}{Q-2}} u(\eta) dV_\theta.
\end{align*}
\end{proof}

\begin{remark}
When $\alpha=2$, \eqref{conformal operator property} is
\begin{equation}\label{conformal 1}
    I_{N,\tilde{\theta},2}(u)\phi(\xi)=\int_N G_\xi^{\theta}(\eta) u(\eta) \phi(\eta)^{\frac{Q+2}{Q-2}} \theta\wedge d\theta^n.
\end{equation}
For given nonnegative function $u$ and taking $\phi(\xi)=\frac{\Phi(\xi)}{I_{N,\tilde{\theta},2}(u)}$, then \eqref{conformal 1} becomes
    $$\Phi(\xi)=\int_N G_\xi^{\theta}(\eta) \tilde{R}(\eta) \Phi(\eta)^{\frac{Q+2}{Q-2}} \theta\wedge d\theta^n \text{\quad with\quad} \tilde{R}(\eta) =u(\eta) I_{N,\tilde{\theta},2}(u)(\eta)^{-\frac{Q+2}{Q-2}},$$
which is the integral type curvature equation \eqref{CR Yamabe integral equ}. Particular, case $u\equiv constant$ is the CR Yamabe problem.

Similarly, \eqref{conformal operator property 1} is corresponding to a class of general integral equation
    $$\Phi(\xi)^{\frac{Q-\alpha}{Q-2}}=\int_N G_\xi^{\theta}(\eta) \tilde{R}(\eta) \Phi(\eta)^{\frac{Q+\alpha}{Q-2}} \theta\wedge d\theta^n.$$
\end{remark}

\end{appendix}

\vskip 1cm
\noindent {\bf Acknowledgements}\\
\noindent
The author would like to thank Professor Meijun Zhu for some helpful discussions and his valuable comments. The project is supported by  the
National Natural Science Foundation of China (Grant No. 11571268,\ 11201443) and Natural Science Foundation of Zhejiang Province (Grant No.  LY18A010013).

\end{document}